%% file: firstcap.tex
\documentclass[11pt]{article}

\title{Examples around the strong Viterbo conjecture}

\author{Jean Gutt\footnote{{\bf{Universit\'e Toulouse III - Paul Sabatier}}, 118 route de Narbonne, 31062 Toulouse Cedex 9, France \& {\bf{Institut National Universitaire Champollion}}, Place de Verdun, 81012 Albi, France} , Michael Hutchings\footnote{{\bf University of California, Berkeley}, partially supported by NSF grant DMS-2005437} \,\,\& Vinicius G. B. Ramos\footnote{{\bf{Instituto de Matem\'atica Pura e Aplicada}}, Estrada Dona Castorina, 110, Rio de Janeiro - RJ - Brasil, 22460-320, partially supported by grants from the Serrapilheira Institute, FAPERJ and CNPq}}

\usepackage{amssymb}
\usepackage{latexsym}
\usepackage{amsmath}
\usepackage{amsthm}
\usepackage{amscd}

\usepackage[dvips]{graphics}
\usepackage{overpic}
\usepackage[mathscr]{euscript}

\newcommand{\mc}[1]{{\mathcal #1}}

\numberwithin{equation}{section}

\newtheorem{theorem}{Theorem}[section]
\newtheorem{proposition}[theorem]{Proposition}
\newtheorem{corollary}[theorem]{Corollary}
\newtheorem{lemma}[theorem]{Lemma}
\newtheorem{lemma-definition}[theorem]{Lemma-Definition}

\newtheorem{conjecture}[theorem]{Conjecture}

\theoremstyle{definition}
\newtheorem{definition}[theorem]{Definition}
\newtheorem{remark}[theorem]{Remark}

\newtheorem{example}[theorem]{Example}

\newcommand{\C}{{\mathbb C}}

\newcommand{\R}{{\mathbb R}}

\newcommand{\Z}{{\mathbb Z}}

\newcommand{\op}{\operatorname}

\newcommand{\bpm}{\begin{pmatrix}}
\newcommand{\epm}{\end{pmatrix}}

\renewcommand{\epsilon}{\varepsilon}

\usepackage[dvips]{graphics}
\usepackage{overpic,subfig}
\usepackage{tikz}
\newcommand{\Hi}{\mathcal{H}}
\newcommand{\Ac}{\mathcal{A}}
\newcommand{\supp}{\op{supp}}

\begin{document}

\maketitle

\begin{abstract}
A strong version of a conjecture of Viterbo asserts that all normalized symplectic capacities agree on convex domains. We review known results showing that certain specific normalized symplectic capacities agree on convex domains. We also review why all normalized symplectic capacities agree on $S^1$-invariant convex domains. We introduce a new class of examples called ``monotone toric domains'', which are not necessarily convex, and which include all dynamically convex toric domains in four dimensions. We prove that for monotone toric domains in four dimensions, all normalized symplectic capacities agree. For monotone toric domains in arbitrary dimension, we prove that the Gromov width agrees with the first equivariant capacity. We also study a family of examples of non-monotone toric domains and determine when the conclusion of the strong Viterbo conjecture holds for these examples. Along the way we compute the cylindrical capacity of a large class of ``weakly convex toric domains'' in four dimensions.
\end{abstract}

\section{Introduction}

If $X$ and $X'$ are domains\footnote{In this paper, a ``domain'' is the closure of an open set. One can of course also consider domains in other symplectic manifolds, but we will not do so here.} in $\R^{2n}=\C^n$, a {\em symplectic embedding\/} from $X$ to $X'$ is a smooth embedding  $\varphi:X\hookrightarrow X'$ such that $\varphi^{\star}\omega=\omega$, where $\omega$ denotes the standard symplectic form on $\R^{2n}$. If there exists a symplectic embedding from $X$ to $X'$, we write $X \underset{s}{\hookrightarrow} X'$.

An important problem in symplectic geometry is to determine when symplectic embeddings exist, and more generally to classify the symplectic embeddings between two given domains. Modern work on this topic began with the Gromov nonsqueezing theorem \cite{Gromov}, which asserts that the ball
\[
B^{2n}(r)=\left\{z\in\C^{n}\;\big|\; \pi|z|^2\le r\right\}
\]
symplectically embeds into the cylinder
\[
Z^{2n}(R) = \left\{z\in\C^n\;\big|\;\pi|z_1|^2\le R\right\}
\]
if and only if $r\leq R$. Many questions about symplectic embeddings remain open, even for simple examples such as ellipsoids and polydisks.

If there exists a symplectic embedding $X \underset{s}{\hookrightarrow} X'$, then we have the volume constraint $\op{Vol}(X)\le \op{Vol}(X')$. To obtain more nontrivial obstructions to the existence of symplectic embeddings, one often uses various symplectic capacities. Definitions of the latter term vary; here we define a \emph{symplectic capacity} to be a function $c$ which assigns to each domain in $\R^{2n}$, possibly in some restricted class, a number $c(X)\in[0,\infty]$, satisfying the following axioms:
\begin{description}
	\item{(Monotonicity)} If $X$ and $X'$ are domains in $\R^{2n}$, and if there exists a symplectic embedding $X \underset{s}{\hookrightarrow} X'$, then $c(X) \le c(X')$.
	\item{(Conformality)} If $r$ is a positive real number then $c(rX) = r^2 c(X)$.	
\end{description}
We say that a symplectic capacity $c$ is {\em normalized\/} if it is defined at least for convex domains and satisfies
\[
c\big(B^{2n}(1)\big)=c\big(Z^{2n}(1)\big)=1.
\]

The first example of a normalized symplectic capacity is the \emph{Gromov width} defined by
\[
	c_{\op{Gr}}(X)=\sup\left\{r\;\bigg|\;B^{2n}(r) \underset{s}{\hookrightarrow} X\right\}.
\]
This trivially satisfies all of the axioms except for the normalization requirement $c_{\op{Gr}}(Z^{2n}(1))$, which holds by Gromov non-squeezing. A similar example is the {\em cylindrical capacity\/} defined by
\[
c_Z(X) = \inf\left\{R\;\bigg|\; X \underset{s}{\hookrightarrow} Z^{2n}(R)\right\}.
\]

Additional examples of normalized symplectic capacities are the Hofer-Zehnder capacity $c_{\op{HZ}}$ defined in \cite{HZ} and the Viterbo capacity $c_{\op{SH}}$ defined in \cite{V}. There are also useful families of symplectic capacities parametrized by a positive integer $k$, including the Ekeland-Hofer capacities $c_k^{\op{EH}}$ defined in \cite{EH,EH2} using calculus of variations; the ``equivariant capacities'' $c_k^{\op{CH}}$ defined in \cite{GuH} using positive equivariant symplectic homology; and in the four-dimensional case, the ECH capacities $c_k^{\op{ECH}}$ defined in \cite{qech} using embedded contact homology. For each of these families, the $k=1$ capacities $c_1^{\op{EH}}$, $c_1^{\op{CH}}$, and $c_1^{\op{ECH}}$ are normalized. Some additional symplectic capacities defined using rational symplectic field theory were recently introduced in \cite{siegel1,siegel2}. For more about symplectic capacities in general we refer to \cite{chls, schlenk} and the references therein.

The goal of this paper is to discuss some results and examples related to the following conjecture, which apparently has been folkore since the 1990s.

\begin{conjecture}[strong Viterbo conjecture]
\label{conj:V}
If $X$ is a convex domain in $\R^{2n}$, then all normalized symplectic capacities of $X$ are equal.
\end{conjecture}

Viterbo conjectured the following statement\footnote{Viterbo also conjectured that equality holds in \eqref{eqn:viterboconj} only if $\op{int}(X)$ is symplectomorphic to an open ball.} in \cite{viterbo-conj}:

\begin{conjecture}[Viterbo conjecture]
\label{eq:viterbo-conj}
If $X$ is a convex domain in $\R^{2n}$ and if $c$ is a normalized symplectic capacity, then
\begin{equation}
\label{eqn:viterboconj}
c(X)\le(n!\op{Vol}(X))^{1/n}.
\end{equation}
\end{conjecture}

The inequality \eqref{eqn:viterboconj} is true when $c$ is the Gromov width $c_{\op{Gr}}$, by the volume constraint, because $\op{Vol}(B^{2n}(r)) = r^n/n!$. Thus Conjecture~\ref{conj:V} implies Conjecture~\ref{eq:viterbo-conj}. The Viterbo conjecture recently gained more attention as it was shown in \cite{AAKO} that it implies the Mahler conjecture\footnote{The Mahler conjecture \cite{Mahler} states that for any $n$-dimensional normed space $V$, we have
	\[
		\op{Vol}(B_V)\op{Vol}(B_{V^*})\geq\frac{4^n}{n!},
	\]
where $B_V$ denotes the unit ball of $V$, and $B_{V^*}$ denotes the unit ball of the dual space $V^*$. For some examples of Conjectures \ref{conj:V} and \ref{eq:viterbo-conj} related to the Mahler conjecture, see \cite{shilu}.} in convex geometry.

\begin{lemma}
\label{lem:restatement}
If $X$ is a domain in $\R^{2n}$, then $c_{\op{Gr}}(X)\le c_Z(X)$, with equality if and only if all normalized symplectic capacities of $X$ agree (when they are defined for $X$).
\end{lemma}

\begin{proof}
It follows from the definitions that if $c$ is a normalized symplectic capacity defined for $X$, then
$
c_{\op{Gr}}(X) \le c(X) \le c_Z(X).
$
\end{proof}

Thus the strong Viterbo conjecture is equivalent to the statement that every convex domain $X$ satisfies $c_{\op{Gr}}(X)=c_Z(X)$. We now discuss some examples where it is known that $c_{\op{Gr}}=c_Z$. Hermann \cite{Hermann} showed that all $T^n$-invariant convex domains have to satisfy $c_{\op{Gr}}=c_Z$. This generalizes to $S^1$-invariant convex domains by the following elementary argument:

\begin{proposition}[Y. Ostrover, private communication]
\label{prop:ostrover}
Let $X$ be a compact convex domain in $\C^{n}$ which is invariant under the $S^1$ action by $e^{i\theta}\cdot z = (e^{i\theta}z_1,\ldots,e^{i\theta}z_n)$. Then $c_{\op{Gr}}(X)=c_Z(X)$.
\end{proposition}

\begin{proof}
By compactness, there exists $z_0\in\partial X$ minimizing the distance to the origin. Let $r>0$ denote this minimal distance. Then the ball $(|z|\le r)$ is contained in $X$, so by definition $c_{\op{Gr}}(X) \ge \pi r^2$.

By applying an element of $U(n)$, we may assume without loss of generality that $z_0=(r,0,\ldots,0)$. By a continuity argument, we can assume without loss of generality that $\partial X$ is a smooth hypersurface in $\R^{2n}$. By the distance minimizing property, the tangent plane to $\partial X$ at $z_0$ is given by $(z\cdot (1,0,\ldots,0)=r)$ where $\cdot$ denotes the real inner product. By convexity, $X$ is contained in the half-space $(z\cdot (1,0,\ldots,0)\le r)$. By the $S^1$ symmetry, $X$ is also contained in the half-space $(z\cdot (e^{i\theta},0,\ldots,0)\le r)$ for each $\theta\in\R/2\pi\Z$. Thus $X$ is contained in the intersection of all these half-spaces, which is the cylinder $|z_1|\le r$. Then $c_Z(X)\le \pi r^2$ by definition.
\end{proof}

\begin{remark}
A similar argument shows that if $k\ge 3$ is an integer and if $X\subset\C^n$  is a convex domain invariant under the $\Z/k$ action by $j\cdot z = (e^{2\pi i j/k}z_1,\ldots,e^{2\pi i j/k}z_n)$, then
\[
\frac{c_Z(X)}{c_{\op{Gr}}(X)} \le \frac{k}{\pi}\tan(\pi/k).
\]
\end{remark}

The role of the convexity hypothesis in Conjecture~\ref{conj:V} is somewhat mysterious. We now explore to what extent non-convex domains can satisfy $c_{\op{Gr}}=c_Z$.

To describe some examples, if $\Omega$ is a domain in $\R^n_{\ge 0}$, define the {\em toric domain\/}
\[
X_\Omega = \left\{z\in\C^n \;\big|\; \pi(|z_1|^2,\ldots,|z_n|^2)\in\Omega\right\}.
\]
The factors of $\pi$ ensure that
\begin{equation}
\label{eqn:voltoric}
\op{Vol}(X_\Omega) = \op{Vol}(\Omega).
\end{equation}
Let $\partial_+\Omega$ denote the set of $\mu\in\partial\Omega$ such that $\mu_j>0$ for all $j=1,\ldots,n$.

\begin{definition}
A {\em monotone toric domain\/} is a compact toric domain $X_\Omega$ with smooth boundary such that if $\mu\in\partial_+\Omega$ and if $v$ an outward normal vector at $\mu$, then $v_j\ge 0$ for all $j=1,\ldots,n$. See Figure \ref{fig:monotonetoric}.

A {\em strictly monotone toric domain\/} is a compact toric domain $X_\Omega$ with smooth boundary such that if $\mu\in\overline{\partial_+\Omega}$ and if $v$ is a nonzero outward normal vector at $\mu$, then $v_j>0$ for all $j=1,\ldots,n$.
\end{definition}

One of our main results is the following:

\begin{theorem}
\label{thm:4d}
(proved in \S\ref{sec:ech})
If $X_\Omega$ is a monotone toric domain in $\R^4$, then $c_{\op{Gr}(X)}=c_Z(X)$.
\end{theorem}

Note that monotone toric domains do not have to be convex; see \S\ref{sec:toric} for details on when toric domains are convex. (Toric domains that are convex are already covered by Proposition~\ref{prop:ostrover}.)
 
To clarify the hypothesis in Theorem~\ref{thm:4d}, let $X$ be a compact domain in $\R^{2n}$ with smooth boundary, and suppose that $X$ is ``star-shaped'', meaning that the radial vector field on $\R^{2n}$ is transverse to $\partial X$. Then there is a well-defined Reeb vector field $R$ on $\partial X$. We say that $X$ is {\em dynamically convex\/} if, in addition to the above hypotheses, every Reeb orbit $\gamma$ has Conley-Zehnder index $\op{CZ}(\gamma)\ge n+1$ if nondegenerate, or in general has minimal Conley-Zehnder index\footnote{If $\gamma$ is degenerate then there is an interval of possible Conley-Zehnder indices of nondegenerate Reeb orbits near $\gamma$ after a perturbation, and for dynamical convexity we require the minimum number in this interval to be at least $n+1$. In the $4$-dimensional case ($n=2$), this means that the dynamical rotation number of the linearized Reeb flow around $\gamma$, which we denote by $\rho(\gamma)\in\R$, is greater than $1$.} at least $n+1$. It was shown by Hofer-Wysocki-Zehnder \cite{HWZ} that if $X$ is strictly convex, then $X$ is dynamically convex. However the Viterbo conjecture implies that not every dynamically convex domain is symplectomorphic to a convex domain; see Remark~\ref{rem:abhs} below.

\begin{proposition}
\label{prop:dyncon}
(proved in \S\ref{sec:toric})
Let $X_\Omega$ be a compact star-shaped toric domain in $\R^4$ with smooth boundary. Then $X_\Omega$ is dynamically convex if and only if $X_\Omega$ is a strictly monotone toric domain.
\end{proposition}

Thus Theorem~\ref{thm:4d} implies that all dynamically convex toric domains in $\R^4$ have $c_{\op{Gr}}=c_Z$.

If $X$ is a star-shaped domain with smooth boundary, let $A_{\min}(X)$ denote the minimal period of a Reeb orbit on $\partial X$.

\begin{remark}
\label{rem:abhs}
Without the toric hypothesis, not all dynamically convex domains in $\R^4$ have $c_{\op{Gr}}=c_Z$. In particular, it is shown in \cite{ABHS} that for $\epsilon>0$ small, there exists a dynamically convex domain $X$ in $\R^4$ such that $A_{\min}(X)^2/(2\op{vol}(X))\ge 2-\epsilon$. One has $c_1^{\op{CH}}(X)\ge A_{\min}(X)$ by \cite[Thm.\ 1.1]{GuH}, and $c_{\op{Gr}}(X)^2\le 2\op{vol}(X)$ by the volume constraint. Thus
\[
\frac{c_Z(X)}{c_{\op{Gr}}(X)} \ge \sqrt{2-\epsilon}.
\] 
\end{remark}

\begin{remark}
It is also not true that all star-shaped toric domains have $c_{\op{Gr}}=c_Z$. Counterexamples have been known for a long time, see e.g.\ \cite{Hermann}, and in \S\ref{sec:ife} we discuss a new family of counterexamples.
\end{remark}

For monotone toric domains in higher dimensions, we do not know how to prove that all normalized symplectic capacities agree, but we can at least prove the following:

\begin{theorem}
\label{thm:arbdim}
(proved in \S\ref{sec:equi})
If $X_\Omega$ is a monotone toric domain in $\R^{2n}$, then
\begin{equation}
\label{eq:mainmoreover}
c_{\op{Gr}}(X_\Omega) = c_1^{\op{CH}}(X_\Omega).
\end{equation}
\end{theorem}

Returning to convex domains, some normalized symplectic capacities are known to agree (not the Gromov width or cylindrical capacity however), as we review in the following theorem:

\begin{theorem}[Ekeland, Hofer, Zehnder, Abbondandolo-Kang, Irie]
\label{thm:equivalence}
If $X$ is a convex domain in $\R^{2n}$, then:
\begin{description}
\item{(a)}
$c_1^{\op{EH}}(X)=c_{\op{HZ}}(X)=c_{\op{SH}}(X)=c_1^{\op{CH}}(X)$.
\item{(b)}
If in addition $\partial X$ is smooth\footnote{Without the smoothness assumption, it is shown in \cite[Prop. 2.7]{AAO} that $c_{\op{HZ}}(X)$ agrees with the minimum action of a ``generalized closed characteristic'' on $\partial X$.}, then all of the capacities in (a) agree with $A_{\op{min}}(X)$.
\end{description}
\end{theorem}

\begin{proof}
Part (b) implies part (a) because every convex domain can be $C^0$ approximated by one with smooth boundary; and the capacities in (a) are $C^0$ continuous functions of the convex domain $X$, by monotonicity and conformality.

Part (b) was shown for $c_{\op{HZ}}(X)$ by Hofer-Zehnder in \cite{HZ}, and for $c_{\op{SH}}(X)$ by Irie \cite{Irie} and Abbondandolo-Kang \cite{AK}. The agreement of these two capacities with $c_1^{\op{CH}}(X)$ for convex domains now follows from the combination of \cite[Theorem 1.24]{GuH} and \cite[Lemma 3.2]{GS}, as explained by Irie in \cite[Remark 2.15]{Irie}. Finally, part (b) for $c_1^{\op{EH}}(X)$ has been claimed and understood for a long time, but since we could not find a complete proof in the literature we give one here in \S\ref{sec:eh}.
\end{proof}

\subsection*{Organization of the paper}

In \S\ref{sec:toric} we discuss different kinds of toric domains and when they are convex or dynamically convex. In \S\ref{sec:equi} we consider the first equivariant capacity and prove Theorem~\ref{thm:arbdim}. In \S\ref{sec:ech} we use ECH capacities to prove Theorem~\ref{thm:4d}. In \S\ref{sec:ife} we consider a family of examples of non-monotone toric domains and determine when they do or do not satisfy the conclusions of Conjectures \ref{conj:V} and \ref{eq:viterbo-conj}. Along the way we compute the cylindrical capacity of a large class of ``weakly convex toric domains'' in four dimensions (Theorem \ref{thm:czwt}). In \S\ref{sec:eh} we review the definition of the first Ekeland-Hofer capacity and complete the (re)proof of Theorem~\ref{thm:equivalence}.

\subsection*{Acknowledgements}
We thank A. Oancea, Y. Ostrover and M. Usher for useful discussions, J. Kang and E. Shelukhin for bringing some parts of the literature to our attention and F. Schlenk for detailed comments on an earlier version of this paper.

\section{Toric domains}
\label{sec:toric}

In this section we review some important classes of toric domains and discuss when they are convex or dynamically convex.

If $\Omega$ is a domain in $\R^n$, define
\[
\widehat{\Omega} = \left\{\mu\in\R^n \;\big|\; (|\mu_1|,\ldots,|\mu_n|)\in\Omega\right\}.
\]

\begin{definition}
\cite{GuH}
A {\em convex toric domain\/} is a toric domain $X_\Omega$ such that $\widehat{\Omega}$ is compact and convex. See Figure \ref{fig:convextoric}.
\end{definition}

This terminology may be misleading because a ``convex toric domain'' is not the same thing as a compact toric domain that is convex in $\R^{2n}$; see Proposition~\ref{prop:convex} below.

\begin{definition}
\cite{GuH}
A {\em concave toric domain\/} is a toric domain $X_\Omega$ such that $\Omega$ is compact and $\R_{\ge 0}^n\setminus\Omega$ is convex. See Figure \ref{fig:concavetoric}.
\end{definition}

We remark that if $X_\Omega$ is a convex toric domain or concave toric domain and if $X_\Omega$ has smooth boundary, then it is a monotone toric domain.

\begin{figure}[ht]
\centering
	\subfloat[A convex toric domain]{
	\begin{minipage}[c][1\width]{0.3\textwidth}
	\centering
	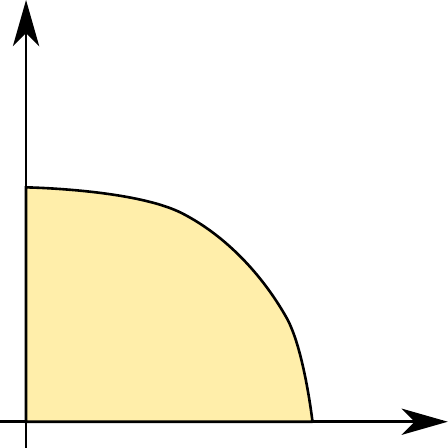\label{fig:convextoric}
	\end{minipage}}
	\qquad\qquad\qquad
	\subfloat[A concave toric domain]{
	\begin{minipage}[c][1\width]{0.3\textwidth}
	\centering
	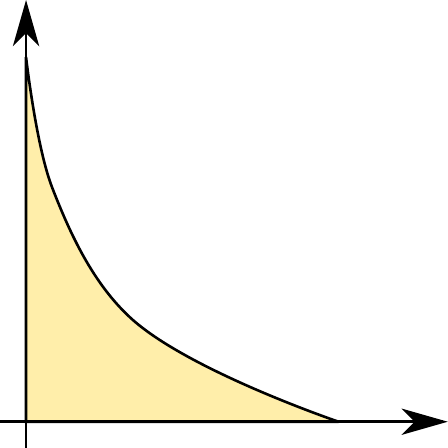\label{fig:concavetoric}
	\end{minipage}}\\
	\subfloat[A monotone toric domain]{
	\begin{minipage}[c][1\width]{0.3\textwidth}
	\centering
	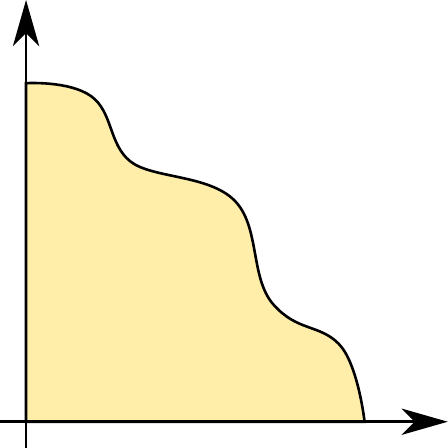\label{fig:monotonetoric}
	\end{minipage}}
	\qquad\qquad\qquad
	\subfloat[A weakly convex toric domain]{
	\begin{minipage}[c][1\width]{0.3\textwidth}
	\centering
	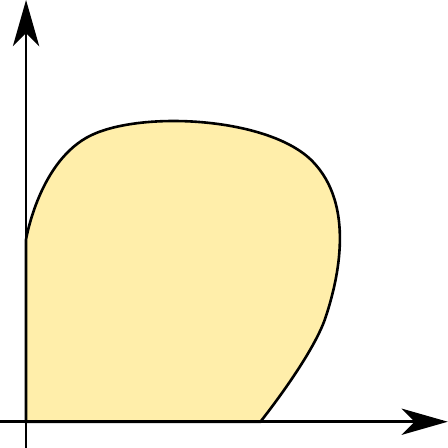\label{fig:wconvextoric}
	\end{minipage}}
	\caption{Examples of toric domains $X_\Omega$ in $\R^4$}
\end{figure}

\begin{proposition}
\label{prop:convex}
A toric domain $X_\Omega$ is a convex subset of $\R^{2n}$ if and only if the set
\begin{equation}
\label{eqn:convex}
\widetilde{\Omega} = 
\left\{\mu\in\R^n \;\bigg|\; \pi\left(|\mu_1|^2,\ldots,|\mu_n|^2\right)\in\Omega\right\}
\end{equation}
is convex in $\R^n$.
\end{proposition}

\begin{proof}
$(\Rightarrow)$
The set $\widetilde{\Omega}$ is just the intersection of the toric domain $X_\Omega$ with the subspace $\R^n\subset\C^n$. If $X_\Omega$ is convex, then its intersection with any linear subspace is also convex.

$(\Leftarrow)$
Suppose that the set $\widetilde{\Omega}$ is convex. Let $z,z'\in X_\Omega$ and let $t\in[0,1]$. We need to show that
\[
(1-t)z + tz'\in X_\Omega.
\]
That is, we need to show that
\begin{equation}
\label{eqn:wnts}
\left(\left|(1-t)z_1+ tz'_1\right| , \ldots , \left|(1-t)z_n+z'_n\right|\right) \in \widetilde{\Omega}.
\end{equation}
We know that the $2^n$ points $(\pm|z_1|,\ldots,\pm|z_n|)$ are all in $\widetilde{\Omega}$, as are the $2^n$ points $(\pm|z_1'|,\ldots,\pm|z_n'|)$. 
By the triangle inequality we have
\[
|(1-t)z_j + tz_j'| \le (1-t)|z_j| + t|z'_j|
\]
for each $j=1,\ldots,n$. It follows that the point in \eqref{eqn:wnts} can be expressed as $(1-t)$ times a convex combination of the points $(\pm|z_1|,\ldots,\pm|z_n|)$, plus $t$ times a convex combination of the points $(\pm|z_1'|,\ldots,\pm|z_n'|)$. Since $\widetilde{\Omega}$ is convex, it follows that \eqref{eqn:wnts} holds.
\end{proof}

\begin{example}
If $X_\Omega$ is a convex toric domain, then $X_\Omega$ is a convex subset of $\R^{2n}$.
\end{example}

\begin{proof}
Similarly to the above argument, this boils down to showing that if $w,w'\in\C$ and $0\le t\le 1$ then
\[
|(1-t)w + tw'|^2 \le (1-t)|w|^2 + t|w'|^2.
\]
The above inequality holds because the right hand side minus the left hand side equals $(t-t^2)|w-w'|^2$.
\end{proof}

However the converse is not true:

\begin{example}
Let $p> 0$, and let $\Omega$ be the positive quadrant of the $L^p$ unit ball,
\[
\Omega = \left\{\mu\in\R^n_{\ge 0} \;\bigg|\; \sum_{j=1}^n\mu_j^p\le 1\right\}.
\]
Then $X_\Omega$ is a concave toric domain if and only if $p\le 1$, and a convex toric domain if and only if $p\ge 1$. By Proposition~\ref{prop:convex}, the domain $X_\Omega$ is convex in $\R^{2n}$ if and only if $p\ge 1/2$.
\end{example}

We now work out when four-dimensional toric domains are dynamically convex.

\begin{proof}[Proof of Proposition \ref{prop:dyncon}.]
As a preliminary remark, note that if a Reeb orbit has rotation number $\rho>1$, then so does every iterate of the Reeb orbit. Thus $X_\Omega$ is dynamically convex if and only if every {\em simple\/} Reeb orbit has rotation number $\rho>1$.

Since $X_\Omega$ is star-shaped, $\Omega$ itself is also star-shaped. Since $X_\Omega$ is compact with smooth boundary, $\overline{\partial_+\Omega}$ is a smooth arc from some point $(0,b)$ with $b>0$ to some point $(a,0)$ with $a>0$.

We can find the simple Reeb orbits and their rotation numbers by the calculations in \cite[\S3.2]{concave} and \cite[\S2.2]{GuH}. The conclusion is the following. There are three types of simple Reeb orbits on $\partial X_\Omega$:
\begin{description}
\item{(i)}
There is a simple Reeb orbit corresponding to $(a,0)$, whose image is the circle in $\partial X_\Omega$ with $\pi|z_1|^2=a$ and $z_2=0$.
\item{(ii)}
Likewise, there is a simple Reeb orbit corresponding to $(0,b)$, whose image is the circle in $\partial X_\Omega$ with $z_1=0$ and $\pi|z_2|^2=b$.
\item{(iii)}
For each point $\mu\in\partial_+\Omega$ where $\partial_+\Omega$ has rational slope, there is an $S^1$ family of simple Reeb orbits whose images sweep out the torus in $\partial X_\Omega$ where $\pi(|z_1|^2,|z_2|^2)=\mu$.
\end{description}
Let $s_1$ denote the slope of $\overline{\partial_+\Omega}$ at $(a,0)$, and let $s_2$ denote the slope of $\overline{\partial_+\Omega}$ at $(0,b)$.
Then the Reeb orbit in (i) has rotation number $\rho = 1-s_1^{-1}$, and the Reeb orbit in (ii) has rotation number $\rho = 1-s_2$. For a Reeb orbit in (iii), let $\nu=(\nu_1,\nu_2)$ be the outward normal vector to $\partial_+\Omega$ at $\mu$, scaled so that $\nu_1,\nu_2$ are relatively prime integers. Then each Reeb orbit in this family has rotation number $\rho=\nu_1+\nu_2$.

If $X_\Omega$ is strictly monotone, then $s_1,s_2<0$, and for each Reeb orbit of type (iii) we have $\nu_1,\nu_2\ge 1$. It follows that every simple Reeb orbit has rotation number $\rho>1$.

Conversely, suppose that every simple Reeb orbit has rotation number $\rho>1$. Applying this to the Reeb orbits (i) and (ii), we obtain that $s_1,s_2<0$. Thus $\partial_+\Omega$ has negative slope near its endpoints. The arc $\overline{\partial_+\Omega}$ can never go horizontal or vertical in its interior, because otherwise there would be a Reeb orbit of type (iii) with $\nu=(1,0)$ or $\nu=(0,1)$, so that $\rho=1$. Thus $X_\Omega$ is strictly monotone.
\end{proof}

\section{The first equivariant capacity}
\label{sec:equi}

We now prove Theorem~\ref{thm:arbdim}. (Some related arguments appeared in \cite[Lem.\ 1.19]{GuH}.) If $a_1,\ldots,a_n >0$, define the ``L-shaped domain''
\[
L(a_1,\ldots,a_n) = \left\{\mu\in\R^n_{\ge 0}\;\big|\;\mbox{$\mu_j\le a_j$ for some $j$}\right\}.
\]

\begin{lemma}
\label{lem:c1chl}
If $a_1,\ldots,a_n> 0$, then
\[
c_1^{\op{CH}}\left(X_{L(a_1,\ldots,a_n)}\right) = \sum_{j=1}^na_j.
\]
\end{lemma}

\begin{proof}
Observe that
\[
\R^n_{\ge 0}\setminus L(a_1,\ldots,a_n) = (a_1,\infty)\times\cdots\times (a_n,\infty).
\]
is convex. Thus $X_{L(a_1,\ldots,a_n)}$ satisfies all the conditions in the definition of ``concave toric domain'', except that it is not compact.

A formula for $c_k^{\op{CH}}$ of a concave toric domain is given in \cite[Thm.\ 1.14]{GuH}. The $k=1$ case of this formula asserts that if $X_\Omega$ is a concave toric domain in $\R^{2n}$, then
\begin{equation}
\label{eqn:c1concave}
c_1^{\op{CH}}(X_\Omega) = \min\left\{\sum_{i=1}^n\mu_i\;\bigg|\;\mu\in\overline{\partial_+\Omega} \right\}.
\end{equation}
By an exhaustion argument (see \cite[Rmk.\ 1.3]{GuH}), this result also applies to $X_{L(a_1,\ldots,a_n)}$. For $\Omega=L(a_1,\ldots,a_n)$, the minimum in \eqref{eqn:c1concave} is realized by $\mu=(a_1,\ldots,a_n)$.
\end{proof}

\begin{lemma}
\label{lem:monl}
If $X_\Omega$ is a monotone toric domain in $\R^{2n}$ and if $\mu\in\partial_+\Omega$, then $\Omega\subset L(\mu_1,\ldots,\mu_n)$.
\end{lemma}

\begin{proof}
By an approximation argument we can assume without loss of generality that $X_\Omega$ is strictly monotone.  Then $\partial_+\Omega$ is the graph of a positive function $f$ over an open set $U\subset\R_{\ge 0}^{n-1}$ with $\partial_jf<0$ for $j=1,\ldots,n-1$.  It follows that if $(\mu_1',\ldots,\mu_{n-1}')\in U$ and $\mu_j'>\mu_j$ for all $j=1,\ldots,n-1$, then $f(\mu_1',\ldots,\mu_{n-1}') < f(\mu_1,\ldots,\mu_{n-1})$. Consequently $\Omega$ does not contain any point $\mu'$ with $\mu'_j>\mu_j$ for all $j=1,\ldots,n$. This means that $\Omega\subset L(\mu_1,\ldots,\mu_n)$.
Figure \ref{fig:embeddings} illustrates this inclusion for $n=2$. 
\end{proof}

\begin{proof}[Proof of Theorem \ref{thm:arbdim}.]
For $a>0$, consider the simplex
\[
\Delta^n(a) = \left\{\mu\in\R^n_{\ge 0}\;\bigg|\; \sum_{j=1}^n\mu_i\le a\right\}.
\]
Observe that the toric domain $X_{\Delta^n(a)}$ is the ball $B^{2n}(a)$. Now let $a>0$ be the largest real number such that $\Delta^n(a)\subset\Omega$; see Figure \ref{fig:embeddings}.

We have $B^{2n}(a)\subset X_\Omega$, so by definition $a\le c_{\op{Gr}}(X_\Omega)$. Since $c_1^{\op{CH}}$ is a normalized symplectic capacity, $c_{\op{Gr}}(X_\Omega) \le c_1^{\op{CH}}(X_\Omega)$. By the maximality property of $a$, there exists a point $\mu\in\overline{\partial_+\Omega}$ with $\sum_{j=1}^n\mu_j=a$. By an approximation argument we can assume that $\mu\in\partial_+\Omega$. By Lemma~\ref{lem:monl}, $X_\Omega\subset X_{L(\mu_1,\ldots,\mu_n)}$. By the monotonicity of $c_1^{\op{CH}}$ and Lemma~\ref{lem:c1chl}, we then have
\[
c_1^{\op{CH}}(X_\Omega) \le c_1^{\op{CH}}\left(X_{L(\mu_1,\ldots,\mu_n)}\right) = \sum_{j=1}^n\mu_j = a.
\]
Combining the above inequalities gives $c_{\op{Gr}}(X_\Omega) = c_1^{\op{CH}}(X_\Omega) = a$.
\end{proof}

\begin{figure}[ht]
\centering
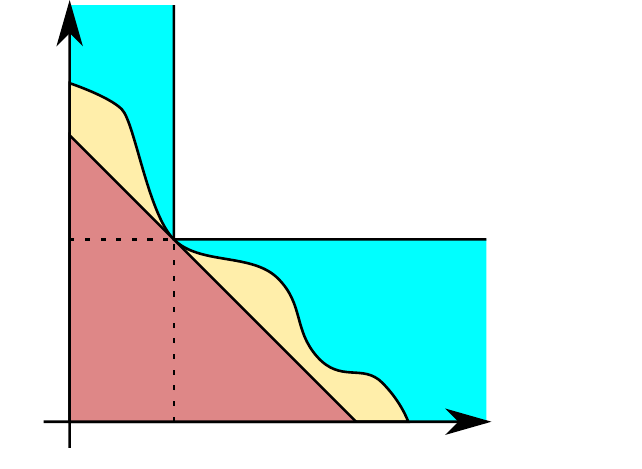
\caption{The inclusions $\Delta^n(a)\subset \Omega\subset L(\mu_1,\dots,\mu_n)$ for $n=2$}
\label{fig:embeddings}
\end{figure}

\section{ECH capacities}
\label{sec:ech}

We now recall some facts about ECH capacities which we will use to prove Theorem~\ref{thm:4d}.

\begin{definition}
\label{def:wctd}
A {\em weakly convex toric domain\/} in $\R^4$ is a compact toric domain $X_\Omega\subset\R^4$ such that $\Omega$ is convex, and $\overline{\partial_+\Omega}$ is an arc with one endpoint on the positive $\mu_1$ axis and one endpoint on the positive $\mu_2$ axis. See Figure \ref{fig:wconvextoric}.
\end{definition}

\begin{theorem}[Cristofaro-Gardiner \cite{concaveconvex}]
\label{thm:concaveconvex}
In $\R^4$, let $X_\Omega$ be a concave toric domain, and let $X_{\Omega'}$ be a weakly convex toric domain. Then there exists a symplectic embedding $\op{int}(X_\Omega) \underset{s}{\hookrightarrow} X_{\Omega'}$ if and only if $c_k^{\op{ECH}}(X_\Omega) \le c_k^{\op{ECH}}(X_{\Omega'})$ for all $k\ge 0$.
\end{theorem}

To make use of this theorem, we need some formulas to compute the ECH capacities $c_k^{\op{ECH}}$. To start, consider a 4-dimensional concave toric domain $X_{\Omega}$. Associated to $X_\Omega$ is a ``weight sequence'' $W(X_\Omega)$, which is a finite or countable multiset of positive real numbers defined in \cite{concave}, see also \cite{bidisk}, as follows. Let $r$ be the largest positive real number such that the triangle $\Delta^2(r)\subset\Omega$. We can write $\Omega\setminus \Delta^2(r)=\widetilde{\Omega}_1\sqcup\widetilde{\Omega}_2$, where $\widetilde{\Omega}_1$ does not intersect the $\mu_2$-axis and $\widetilde{\Omega}_2$ does not intersect the $\mu_1$-axis. It is possible that $\widetilde{\Omega}_1$ and/or $\widetilde{\Omega}_2$ is empty. After translating the closures of $\widetilde{\Omega}_1$ or $\widetilde{\Omega}_2$ by $(-r,0)$ and $(0,-r)$ and multiplying them by the matrices $\begin{bmatrix}1&1\\0&1\end{bmatrix}$ and $\begin{bmatrix}1&0\\1&1\end{bmatrix}$, respectively, we obtain two new domains $\Omega_1$ and $\Omega_2$ in $\R_{\ge 0}^2$ such that $X_{\Omega_1}$ and $X_{\Omega_2}$ are concave toric domains. We then inductively define
\begin{equation}
\label{eqn:weights}
W(X_\Omega) = (r) \cup W(X_{\Omega_1}) \cup W(X_{\Omega_2}),
\end{equation}
where `$\cup$' denotes the union of multisets, and the term $W(X_{\Omega_i})$ is omitted if $\Omega_i$ is empty.

Let us call two subsets of $\R^2$ ``affine equivalent'' if one can be obtained from the other by the composition of a translation and an element of $\op{GL}(2,\Z)$. If $W(X_{\Omega})=(a_1,a_2,\ldots)$, then the domain $\Omega$ is canonically decomposed into triangles, which are affine equivalent to the triangles $\Delta^2(a_1), \Delta^2(a_2), \ldots$ and which meet only along their edges; the first of these triangles is $\Delta^2(r)$. See \cite[\S3.1]{ruelle} for more details. We now recall the ``Traynor trick'':

\begin{proposition}
\label{prop:traynor}
\cite{traynor}
If $T\subset\R^2_{\ge 0}$ is a triangle affine equivalent to $\Delta^2(a)$, then there is a symplectic embedding $\op{int}(B^4(a))\underset{s}{\hookrightarrow} X_{\op{int}(T)}$.
\end{proposition}

\noindent
As a result, there is a symplectic embedding
\[
\coprod_i\op{int}(B^4(a_i))\subset X_\Omega.
\]
Consequently, by the monotonicity property of ECH capacities, we have
\begin{equation}
\label{eqn:echconcave}
c_k^{\op{ECH}}\left(\coprod_i\op{int}(B^4(a_i))\right) \le c_k^{\op{ECH}}(X_\Omega).
\end{equation}

\begin{theorem}[\cite{concave}]
\label{thm:concave}
If $X_\Omega$ is a four-dimensional concave toric domain with weight expansion $W(X_\Omega)=(a_1,a_2,\ldots)$, then equality holds in \eqref{eqn:echconcave}.
\end{theorem}

To make this more explicit, we know from \cite{qech} that\footnote{For the sequence of numbers $a_i$ coming from a weight expansion, or for any finite sequence, the supremum in \eqref{eqn:ckechunion} is achieved, so we can write `max' instead of `sup'.}
\begin{equation}
\label{eqn:ckechunion}
c_k^{\op{ECH}}\left(\coprod_i\op{int}(B^4(a_i))\right) = \sup_{k_1+\cdots = k}\sum_ic_{k_i}^{\op{ECH}}(\op{int}(B^4(a_i)))
\end{equation}
and
\begin{equation}
\label{eqn:ckechball}
c_k^{\op{ECH}}(\op{int}(B^4(a))) = c_k^{\op{ECH}}(B^4(a)) = da,
\end{equation}
where $d$ is the unique nonnegative integer such that
\[
d^2+d \le 2k \le d^2+3d.
\]

To state the next lemma, given $a_1,a_2>0$, define the polydisk
\[
P(a_1,a_2) = \left\{z\in\C^2\;\bigg|\;\pi|z_1|^2\le a_1, \; \pi|z_2|^2\le a_2\right\}.
\]
This is a convex toric domain $X_{\Omega'}$ where $\Omega'$ is a rectangle of side lengths $a_1$ and $a_2$.

\begin{lemma}
\label{lem:folding}
Let $X_\Omega$ be a four-dimensional concave toric domain. Let $(a,0)$ and $(0,b)$ be the points where $\overline{\partial_+\Omega}$ intersects the axes. Let $\mu$ be a point on $\overline{\partial_+\Omega}$ minimizing $\mu_1+\mu_2$, and write $r=\mu_1+\mu_2$. Then there exists a symplectic embedding
\[
\op{int}(X_\Omega) \underset{s}{\hookrightarrow} P(r,\max(b,a-r)).
\]
\end{lemma}

\begin{proof}
One might hope for a direct construction using some version of ``symplectic folding'' \cite{folding}, but we will instead use the above ECH machinery. By Theorem~\ref{thm:concaveconvex}, it is enough to show that
\begin{equation}
\label{eqn:ccns}
c_k^{\op{ECH}}(X_\Omega) \le c_k^{\op{ECH}}(P(r,\max(b,a-r))
\end{equation}
for each nonnegative integer $k$. 

Consider the weight expansion $W(X_\Omega)=(a_1,a_2,\ldots)$ where $a_1=r$. The decomposition of $\Omega$ into triangles corresponding to the weight expansion consists of the triangle $\Delta^2(r)$, plus some additional triangles in the triangle with corners $ (0,r), (\mu_1,\mu_2), (0,b)$, plus some additional triangles in the triangle with corners $(\mu_1,\mu_2), (r,0), (a,0)$; see Figure~\ref{fig:weight}. The latter triangle is affine equivalent to the triangle with corners $(\mu_1,\mu_2), (r,0), (r,a-r)$; see Figure \ref{fig:pack}. This allows us to pack triangles affine equivalent to $\Delta^2(a_1), \Delta^2(a_2), \ldots$ into the rectangle with horizontal side length $r$ and vertical side length $\max(b,a-r)$. Thus by the Traynor trick, we have a symplectic embedding
\[
\coprod_i\op{int}(B(a_i)) \underset{s}{\hookrightarrow} P(r,\max(b,a-r)).
\]
Then Theorem~\ref{thm:concave} and the monotonicity of ECH capacities imply \eqref{eqn:ccns}.
\end{proof}

	\begin{figure}[ht]
	 \subfloat[Weights of $X_\Omega$]{
	\begin{minipage}[c][0.8\width]{0.4\textwidth}
	\centering
	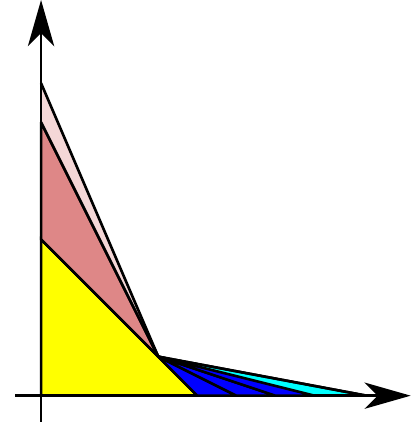\label{fig:weight}
	\end{minipage}}
	\hfill
	\subfloat[Ball packing into a polydisk]{
	\begin{minipage}[c][0.8\width]{0.4\textwidth}
	\centering
	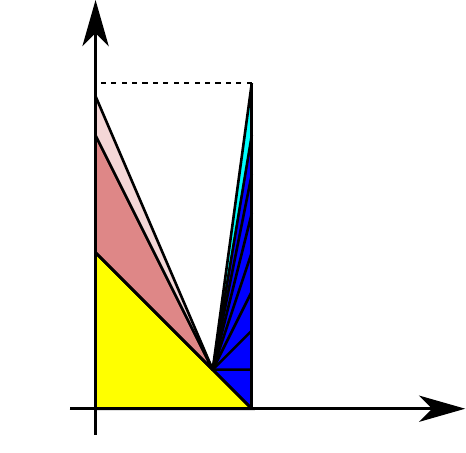\label{fig:pack}
	\end{minipage}}
	\caption{Embedding a concave toric domain into a polydisk}
	\end{figure}

\begin{proof}[Proof of Theorem \ref{thm:4d}]
Let $r$ be the largest positive real number such that $\Delta^2(r)\subset\Omega$. We have $B^4(r)\subset X_\Omega$, so $r\le c_{\op{Gr}}(X_\Omega)$, and we just need to show that $c_Z(X_\Omega)\le r$.

Let $\mu$ be a point on $\partial_+\Omega$ such that $\mu_1+\mu_2=r$. By an approximation argument, we can assume that $X_\Omega$ is strictly monotone, so that the tangent line to $\partial_+\Omega$ at $\mu$ is not horizontal or vertical. Then we can find $a,b>r$ such that $\Omega$ is contained in the quadrilateral with vertices $(0,0)$, $(a,0)$, $(\mu_1,\mu_2)$, and $(0,b)$. It then follows from Lemma~\ref{lem:folding} that there exists a symplectic embedding $\op{int}(X_\Omega)\underset{s}{\hookrightarrow} P(r,R)$ for some $R>0$. Since $P(r,R)\subset Z^4(r)$, it follows that $c_Z(X_\Omega) \le r$.
\end{proof}

\section{A family of non-monotone toric examples}
\label{sec:ife}

We now study a family of examples of non-monotone toric domains, and we determine when they satisfy the conclusions of Conjecture \ref{conj:V} or Conjecture \ref{eq:viterbo-conj}.

For $0<a<1/2$, let $\Omega_a$ be the convex polygon with corners $(0,0)$, $(1-2a,0)$, $(1-a,a)$, $(a,1-a)$ and $(0,1-2a)$, and write $X_a=X_{\Omega_a}$; see Figure \ref{fig:example}. Then $X_a$ is a weakly convex (but not monotone) toric domain.

\begin{proposition}
\label{prop:ife}
Let $0< a<1/2$. Then the Gromov width and cylindrical capacity of $X_a$ are given by
\begin{align}
\label{eqn:cgrxa}
c_{\op{Gr}}(X_a) &= \min(1-a,2-4a),\\
\label{eqn:czxa}
c_Z(X_a) &= 1-a.
\end{align}
\end{proposition}

\begin{corollary}
\label{cor:ife}
Let $0<a<1/2$ and let $X_a$ be as above. Then:
\begin{description}
\item{(a)} The conclusion of Conjecture \ref{conj:V} holds for $X_a$, i.e.\  all normalized symplectic capacities defined for $X_a$ agree, if and only if $a\le 1/3$.
\item{(b)} The conclusion of Conjecture \ref{eq:viterbo-conj} holds for $X_a$, i.e.\ every normalized symplectic capacity $c$ defined for $X_a$ satisfies $c(X_a) \le \sqrt{2\op{Vol}(X_a)}$, if and only if $a\le 2/5$.
\end{description}
\end{corollary}

\begin{proof}[Proof of Corollary \ref{cor:ife}]
(a) By Lemma~\ref{lem:restatement}, we need to check that $c_{\op{Gr}}(X_a) = c_Z(X_a)$ if and only if $a\le 1/3$. This follows directly from \eqref{eqn:cgrxa} and \eqref{eqn:czxa}.

(b) Since $c_Z$ is the largest normalized symplectic capacity, the conclusion of Conjecture \ref{eq:viterbo-conj} holds for $X_a$ if and only if
\begin{equation}
\label{eqn:ve}
c_Z(X_a) \le \sqrt{2\op{Vol}(X_a)}.
\end{equation}
By equation \eqref{eqn:voltoric}, we have
\[
\op{Vol}(X_{\Omega_a}) = \frac{1-4a^2}{2}.
\]
It follows from this and \eqref{eqn:czxa} that \eqref{eqn:ve} holds if and only if $a\le 2/5$.
\end{proof}

\begin{remark}
To recap, the conclusion of Conjecture \ref{conj:V} holds if and only if the ratio $c_{Z}/c_{\op{Gr}}=1$, and the conclusion of Conjecture \ref{eq:viterbo-conj} holds if and only if the ratio $c_Z^n/(n!\op{Vol})\le 1$. The above calculations show that both of these ratios for $X_a$ go to infinity as $a\to 1/2$.
\end{remark}

\begin{figure}[th]
\centering
\subfloat[The domain $\Omega_a$]{
	\begin{minipage}[c][0.9\width]{0.4\textwidth}
	\centering
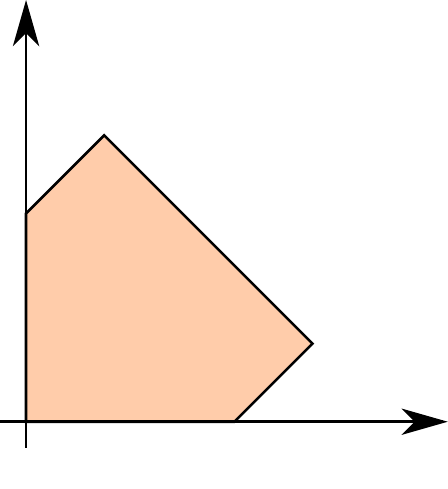\label{fig:example}
\end{minipage}}
\hfill
\subfloat[A domain to which Theorem \ref{thm:czwt} applies]{
	\begin{minipage}[c][0.9\width]{0.4\textwidth}
	\centering
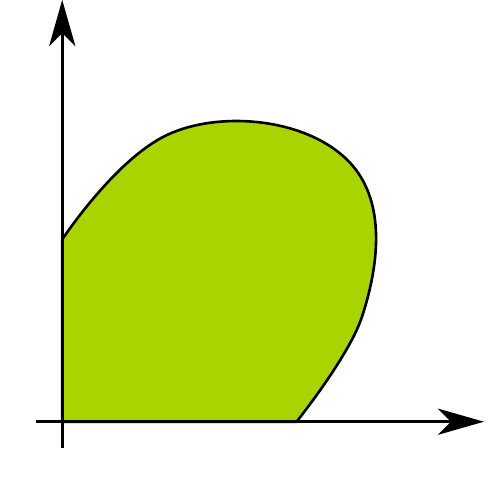\label{fig:czwt}
\end{minipage}}
\caption{Some domains}
\end{figure}

To prove Proposition \ref{prop:ife}, we will use the following formula for the ECH capacities of a weakly convex toric domain $X_\Omega$. Let $r$ be the smallest positive real number such that $\Omega\subset\Delta^2(r)$. Then $\Delta^2(r)\setminus\Omega = \widetilde{\Omega}_1\sqcup \widetilde{\Omega}_2$ where $\widetilde{\Omega}_1$ does not intersect the $\mu_2$-axis, and $\widetilde{\Omega}_2$ does not intersect the $\mu_1$-axis. It is possible that $\widetilde{\Omega}_1$ and/or $\widetilde{\Omega}_2$ is empty. As in the discussion preceding \eqref{eqn:weights}, the closures of $\widetilde{\Omega}_1$ and $\widetilde{\Omega}_2$ are affine equivalent to domains $\Omega_1$ and $\Omega_2$ such that $X_{\Omega_1}$ and $X_{\Omega_2}$ are concave toric domains. Denote the union (as multisets) of their weight sequences by
\[
W(X_{\Omega_1}) \cup W(X_{\Omega_2}) = (a_1,\ldots).
\]
We then have:

\begin{theorem}[Choi--Cristofaro-Gardiner \cite{concaveconvex}]
\label{thm:weaklyconvex}
If $X_\Omega$ is a four-dimensional weakly convex toric domain as above, then
\begin{equation}
\label{eqn:weaklyconvex}
c_k^{\op{ECH}}(X_\Omega) = \inf_{l\ge 0}\left\{c_{k+l}^{\op{ECH}}\left(B^4(r)\right) - c_l^{\op{ECH}}\left(\coprod_iB^4(a_i)\right)\right\}.
\end{equation}
\end{theorem}

We need one more lemma, which follows from \cite[Cor.\ 4.2]{LMS}:

\begin{lemma}
\label{lem:diamond}
Let $\mu_1,\mu_2\ge a > 0$. Let $\Omega$ be the ``diamond'' in $\R^2_{\ge 0}$ given by the convex hull of the points $(\mu_1\pm a,\mu_2)$ and $(\mu_1,\mu_2\pm a)$. Then there is a symplectic embedding
\[
\op{int}(B^4(2a)) \underset{s}{\hookrightarrow} X_\Omega.
\]
\end{lemma}

\begin{proof}[Proof of Proposition \ref{prop:ife}]
To prove \eqref{eqn:cgrxa}, we first describe the ECH capacities of $X_a$. In the formula \eqref{eqn:weaklyconvex} for $X_a$, we have $r=1$, while the weight expansions of $\Omega_1$ and $\Omega_2$ are both $(a,a)$; the corresponding triangles are shown in Figure \ref{fig:balls}(b). Thus by Theorem \ref{thm:weaklyconvex} and equation \eqref{eqn:ckechunion}, we have
\begin{equation}
\label{eqn:ckecha}
c_k^{\op{ECH}}(X_a) = \inf_{l_1,\ldots,l_4\ge 0}\left\{c_{k+l_1+l_2+l_3+l_4}^{\op{ECH}}\left(B^4(1)\right) - \sum_{i=1}^4 c_{l_i}^{\op{ECH}}\left(B^4(a)\right)\right\}.
\end{equation}
We also note from \eqref{eqn:ckechball} that
\[
c_1^{\op{ECH}}(B^4(r)) = c_2^{\op{ECH}}(B^4(r))=r, \quad\quad\quad c_5^{\op{ECH}}(B^4(r))=2r.
\]

Taking $k=1$ and $(l_1,\ldots,l_4) = (1,0,0,0)$ in equation \eqref{eqn:ckecha}, we get
\begin{equation}
\label{eqn:c1echa}
c_1^{\op{ECH}}(X_{\Omega_a}) \le 1-a.
\end{equation}
Taking $k=1$ and $(l_1,\ldots,l_4)=(1,1,1,1)$ in equation \eqref{eqn:ckecha}, we get
\begin{equation}
\label{eqn:c5echa}
c_1^{\op{ECH}}(X_{\Omega_a}) \le 2 - 4a.
\end{equation}
By \eqref{eqn:c1echa} and \eqref{eqn:c5echa} and the fact that $c_1^{\op{ECH}}$ is a normalized symplectic capacity, we conclude that
\begin{equation}
\label{eqn:mina}
c_{\op{Gr}}(X_{\Omega_a}) \le \min(1-a,2-4a).
\end{equation}

To prove the reverse inequality to \eqref{eqn:mina}, suppose first that $0<a\le 1/3$. It is enough to prove that there exists a symplectic embedding $\op{int}(B^4(1-a)) \underset{s}{\hookrightarrow} X_{\Omega_a}$. By Theorem~\ref{thm:concaveconvex}, it is enough to show that
\[
c_k^{\op{ECH}}(B^4(1-a)) \le c_k^{\op{ECH}}(X_{\Omega_a})
\]
for all nonnegative integers $k$. By equation \eqref{eqn:ckecha}, the above inequality is equivalent to
\begin{equation}
\label{eqn:aiet}
c_k^{\op{ECH}}(B^4(1-a)) + \sum_{i=1}^4c_{l_i}^{\op{ECH}}(B^4(a)) \le c_{k+l_1+l_2+l_3+l_4}^{\op{ECH}}(B^4(1))
\end{equation}
for all nonnegative integers $k,l_1,\ldots,l_4\ge 0$. To prove \eqref{eqn:aiet}, by the monotonicity of ECH capacities and the disjoint union formula \eqref{eqn:ckechunion}, it suffices to find a symplectic embedding
\[
\op{int}\left(B^4(1-a) \sqcup \coprod_4 B^4(a)  \right) \underset{s}{\hookrightarrow} B^4(1).
\]
This embedding exists by the Traynor trick (Proposition \ref{prop:traynor}) using the triangles shown in Figure \ref{fig:balls}(a).

Finally, when $1/3\le a <1/2$, it is enough to show that there exists a symplectic embedding $\op{int}(B^4(2-4a)) \underset{s}{\hookrightarrow} X_{\Omega_a}$.  This exists by Lemma~\ref{lem:diamond} using the diamond shown in Figure \ref{fig:balls}(b).

This completes the proof of \eqref{eqn:cgrxa}. Equation \eqref{eqn:czxa} follows from Theorem~\ref{thm:czwt} below.
\end{proof}

\begin{figure}
\centering
\subfloat[$0<a\le 1/3$]{
	\begin{minipage}[c][0.9\width]{0.4\textwidth}
\centering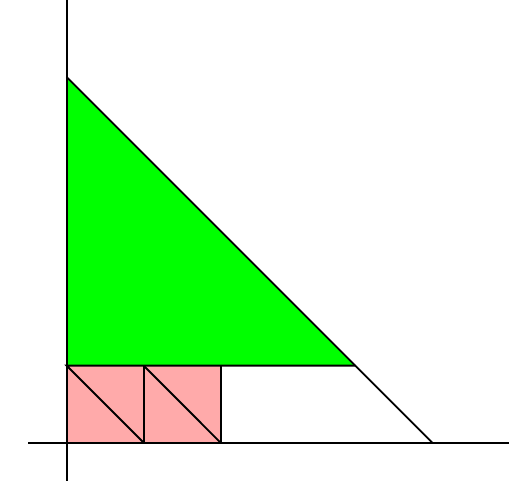
\end{minipage}}
\hfill
\subfloat[$1/3\le a < 1/2$]{
	\begin{minipage}[c][0.9\width]{0.4\textwidth}
\centering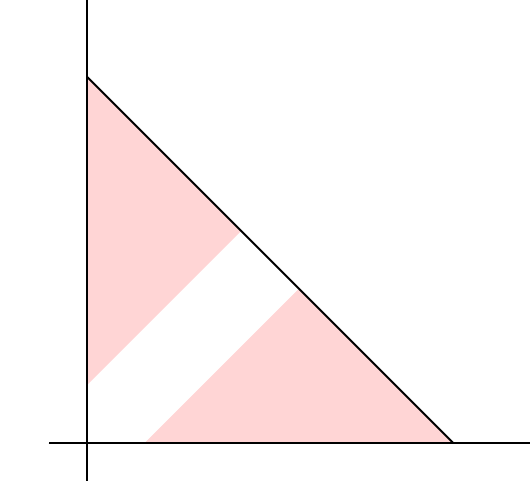
\end{minipage}}
\caption{Ball packings}
\label{fig:balls}
\end{figure}

\begin{theorem}
\label{thm:czwt}
Let $X_\Omega\subset\R^4$ be a weakly convex toric domain, see Definition \ref{def:wctd}.  For $j=1,2$, let
\[
M_j=\max\{\mu_j\mid \mu\in\Omega\}.
\]
Assume that there exists $(M_1,\mu_2)\in\overline{\partial_+\Omega}$ with $\mu_2\le M_1$, and that there exists $(\mu_1,M_2)\in\overline{\partial_+\Omega}$ with $\mu_1\le M_2$. Then
\[
c_Z(X_\Omega) = \min(M_1,M_2).
\]
\end{theorem}

That is, under the hypotheses of the theorem, see Figure \ref{fig:czwt}, an optimal symplectic embedding of $X_\Omega$ into a cylinder is given by the inclusion of $X_\Omega$ into $(\pi|z_1|^2\le M_1)$ or $(\pi|z_2|^2\le M_2)$.

\begin{proof}
From the above inclusions we have $c_Z(X_\Omega) \le \min(M_1,M_2)$. To prove the reverse inequality, suppose that there exists a symplectic embedding
\begin{equation}
\label{eqn:beyond}
X_\Omega\underset{s}{\hookrightarrow} Z^4(R).
\end{equation}
We need to show that $R\ge \min(M_1,M_2)$. To do so, we will use ideas\footnote{The main theorem in \cite{beyond} gives a general obstruction to a symplectic embedding of one four-dimensional convex toric domain into another, which sometimes goes beyond the obstruction coming from ECH capacities. This theorem can be generalized to weakly convex toric domains; but rather than carry out the full generalization, we will just explain the simple case of this that we need.} from \cite{beyond}.

Let $\epsilon>0$ be small. Let $(A,0)$ and $(0,B)$ denote the endpoints of $\overline{\partial_+\Omega}$. By an approximation argument, we can assume that $\overline{\partial_+\Omega}$ is smooth, and that $\partial_+\Omega$ has positive slope less than $\epsilon$ near $(A,0)$ and slope greater than $\epsilon^{-1}$ near $(0,B)$. As in the proof of Proposition~\ref{prop:dyncon}, there are then three types of Reeb orbits on $\partial X_\Omega$:
\begin{description}
\item{(i)}
There is a simple Reeb orbit whose image is the circle with $\pi|z_1|^2=A$ and $z_2=0$. This Reeb orbit has symplectic action (period) equal to $A$, and rotation number $1-\epsilon^{-1}$.
\item{(ii)}
There is a simple Reeb orbit whose image is the circle with $z_1=0$ and $\pi|z_2|^2=B$. This Reeb orbit has symplectic action $B$ and rotation number $1-\epsilon^{-1}$.
\item{(iii)}
For each point $\mu\in\partial_+\Omega$ where $\partial_+\Omega$ has rational slope, there is an $S^1$ family of simple Reeb orbits in the torus where $\pi(|z_1|^2,|z_2|^2)=\mu$. If $\nu=(\nu_1,\nu_2)$ is the outward normal vector to $\partial_+\Omega$ at $\mu$, scaled so that $\nu_1,\nu_2$ are relatively prime integers, then these Reeb orbits have rotation number $\nu_1+\nu_2$ and symplectic action $\mu\cdot\nu$. See \cite[\S2.2]{GuH}.
\end{description}

We claim now that:
\begin{description}
\item{(*)}
Every Reeb orbit on $\partial X_\Omega$ with positive rotation number has symplectic action at least $\min(M_1,M_2)$.
\end{description}
To prove this claim, we only need to check the type (iii) simple Reeb orbits where $\nu_1+\nu_2\ge 1$. For such an orbit we must have $\nu_1\ge 1$ or $\nu_2\ge 1$. Suppose first that $\nu_1 \ge 1$. By the hypotheses of the theorem there exists $\mu_2'$ such that $(M_1,\mu_2')\in\overline{\partial_+\Omega}$ and $M_1\ge \mu_2'$. Since $\Omega$ is convex and $\nu$ is an outward normal at $\mu$, the symplectic action
\[
\mu\cdot \nu \ge (M_1,\mu_2')\cdot\nu  = M_1 + (\nu_1-1)(M_1-\mu_2') + (\nu_1+\nu_2-1)\mu_2'   \ge M_1.
\]
Likewise, if $\nu_2 \ge 1$, then the symplectic action $\mu\cdot\nu\ge M_2$.

As in \cite[\S5.3]{beyond}, starting from the symplectic embedding \eqref{eqn:beyond}, by replacing $X_\Omega$ with an appropriate subset and replacing $Z^4(R)$ with an appropriate superset, we obtain a symplectic embedding $X'\underset{s}{\hookrightarrow} \op{int}(Z')$, where:
\begin{itemize}
\item $Z'$ is an ellipsoid whose boundary has one simple Reeb orbit $\gamma_+$ with symplectic action $\mc{A}(\gamma_+)=R+\epsilon$ and Conley-Zehnder index $\op{CZ}(\gamma_+)=3$, another simple Reeb orbit with very large symplectic action, and no other simple Reeb orbits.
\item $X'$ is a (non-toric) star-shaped domain with smooth boundary, all of whose Reeb orbits are nondegenerate. Every Reeb orbit on $\partial X'$ with rotation number greater than or equal to $1$ has action at least $\min(M_1,M_2)-\epsilon$.
\end{itemize}

The symplectic embedding gives rise to a strong symplectic cobordism $W$ whose positive boundary is $\partial Z'$ and whose negative boundary is $\partial X'$. The argument in \cite[\S6]{beyond} shows that for a generic ``cobordism-admissible'' almost complex structure $J$ on the ``completion'' of $W$, there exists an embedded $J$-holomorphic curve $u$ with one positive end asymptotic to the Reeb orbit $\gamma_+$ in $\partial Z'$, negative ends asymptotic to some Reeb orbits $\gamma_1,\ldots,\gamma_m$ in $\partial X'$, and Fredholm index $\op{ind}(u)=0$. The Fredholm index is computed by the formula
\begin{equation}
\label{eqn:fredholmindex}
\op{ind}(u) = 2g + \left[\op{CZ}(\gamma_+)-1\right] - \sum_{i=1}^m\left[\op{CZ}(\gamma_i)-1\right]
\end{equation}
where $g$ denotes the genus of $u$. Furthermore, since $J$-holomorphic curves decrease symplectic action, we have
\begin{equation}
\label{eqn:decreaseaction}
\mc{A}(\gamma_+) \ge \sum_{i=1}^m\mc{A}(\gamma_i).
\end{equation}

We claim now that at least one of the Reeb orbits $\gamma_i$ has action at least $\min(M_1,M_2)-\epsilon$. Then the inequality \eqref{eqn:decreaseaction} gives
\[
R+\epsilon \ge \min(M_1,M_2) - \epsilon,
\]
and since $\epsilon>0$ was arbitrarily small, we are done.

To prove the above claim, suppose to the contrary that all of the Reeb orbits $\gamma_i$ have action less than $\min(M_1,M_2)-\epsilon$. Then all of the Reeb orbits $\gamma_i$ have rotation number $\rho(\gamma_i)<1$, which means that they all have Conley-Zehnder index $\op{CZ}(\gamma_i)\le 1$. It now follows from \eqref{eqn:fredholmindex} that $\op{ind}(u)\ge 2$, which is a contradiction\footnote{One way to think about the information that we are getting out of \eqref{eqn:fredholmindex}, as well as the general symplectic embedding obstruction in \cite{beyond}, is that we are making essential use of the fact that every holomorphic curve has nonnegative genus.}.
\end{proof}

\section{The first Ekeland-Hofer capacity}
\label{sec:eh}

The goal of this section is to (re)prove the following theorem. This is well-known in the community and is attributed to Ekeland, Hofer and Zehnder \cite{EH2, HZ87}. It was first mentioned by Viterbo in \cite[Proposition 3.10]{viterbocapacite}.

\begin{theorem}[Ekeland-Hofer-Zehnder]
\label{thm:ehz}
Let $W\subset\R^{2n}$ be a compact convex domain with smooth boundary. Then
\[
c_1^{\op{EH}}(W) = A_{\min}(W).
\]
\end{theorem}

We start by recalling the definition of the first Ekeland-Hofer capacity $c_1^{\op{EH}}$. Let $E=H^{1/2}(S^1,\R^{2n})$. That is, if $x\in L^2(S^1,\R^{2n})$ is written as a Fourier series $x=\sum_{k\in\Z} e^{2\pi ikt} x_k$ where $x_k\in\R^{2n}$, then 
\[
x\in E\iff \sum_{k\in\Z} |k||x_k|^2<\infty.
\]
Recall that there is an orthogonal splitting $E=E^+\oplus E^0\oplus E^-$ and orthogonal projections $P^\circ:E\to E^\circ$ where $\circ=+,0,-$. The symplectic action of $x\in E$ is defined to be
\[
A(x)=\frac{1}{2}\left(\Vert P^+ x\Vert_{H^{1/2}}^2-\Vert P^- x\Vert_{H^{1/2}}^2\right).
\]
It follows from a simple calculation that if $x$ is smooth, then $A(x)=\int_x \lambda_0$, where $\lambda_0$ denotes the standard Liouville form on $\R^{2n}$.

Let $\Hi$ denote the set of $H\in C^{\infty}(\R^{2n})$ such that
\begin{itemize}
\item $H|_U\equiv 0$ for some $U\subset\R^{2n}$ open,
\item $H(z)=c|z|^2$ for $z>>0$ where $c\not\in \{\pi,2\pi,3\pi,\dots\}$.
\end{itemize}
For $H\in\Hi$, the action functional $\Ac_H:H^{1/2}(S^1,\R^{2n})\to\R$ is defined by
\begin{equation}\label{eq:ach}\Ac_H(x)=A(x)-\int_0^1 H(x(t))dt.\end{equation}
Note that the natural action of $S^1$ on itself induces an $S^1$-action on $E$. Let $\Gamma$ be the set of homeomorphisms $h:E\to E$ such that $h$ can be written as
\[h(x)=e^{\gamma_+(x)}P^+x+P^0x+e^{\gamma_-(x)}P^-x+K(x),\]
where $\gamma_+,\gamma_-:E\to \R$ are continuous, $S^1$-invariant and map bounded sets to bounded sets, and $K:E\to E$ is continuous, $S^1$-equivariant and maps bounded sets to precompact sets.
Let $S^+$ denote the unit sphere in $E^+$ with respect to the $H^{1/2}$ norm. The first Ekeland-Hofer capacity is defined in \cite{EH2} by
\[
c_{1}^{\op{EH}}(W)=\inf\{c_{H,1}\mid H\in\Hi,W\subset\supp H\},
\]
where
\[
c_{H,1}=\inf\{\sup \Ac_H(\xi)\mid \xi\subset E\text{ is $S^1$-invariant, and }\forall h\in\Gamma:h(\xi)\cap S^+\neq \emptyset\}.
\]

\begin{proof}[Proof of Theorem \ref{thm:ehz}.]
Since $W$ is star-shaped, there is a unique differentiable function $r:\R^{2n}\to \R$ which is $C^\infty$ in $\R^{2n}\setminus\{0\}$ satisfying $r(cz)=c^2r(z)$ for $c\ge 0$ such that
\[
\begin{aligned}W&=\{z\in\R^{2n}\mid r(z)\le 1\},\\
\partial W&= \{z\in\R^{2n}\mid r(z)= 1\}.
\end{aligned}
\]
Let $\alpha=A_{\min}(W)$ and fix $\epsilon>0$. Let $f\in C^{\infty}_{\ge 0} (\R)$ be a convex function such that $f(r)=0$ for $r\le 1$ and $f(r)=Cr-(\alpha+\epsilon)$ for $r\ge 2$ for some constant $C>\alpha$. In particular,
\begin{equation}\label{eq:f}
f(r)\ge Cr-(\alpha+\epsilon),\quad\text{for all } r.\end{equation}
We now choose a convex function $H\in C^\infty(\R^{2n})$ such that
\begin{equation}
\begin{array}{rcll}\label{eq:hr}
H(z)&=&f(r(z)),&\text{ if }r(z)\le 2,\\
H(z)&\ge&f(r(z)),&\text{ for all }z\in\R^{2n},\\
H(z)&=&c\,|z|^2,&\text{ if }z>>0\text{ for some }c\in\R_{>0}\setminus\pi\Z.
\end{array}
\end{equation}
Let $x_0\in E$ be an action-minimizing Reeb orbit on $\partial W$, reparametrized as a map $x_0:\R/\Z=S^1\to\R^{2n}$ of speed $\alpha$, so that $A(x_0)=\alpha$ and $r(x_0)\equiv 1$ and $\dot{x}_0=\alpha J\nabla r(x_0)$. From a simple calculation we deduce that $x_0$ is a critical point of the functional $\Psi:E\to\R$ defined by
\begin{equation}
\label{eq:psi}
\Psi(x)=A(x)-\alpha\int_0^1 r(x(t))\,dt.
\end{equation}
Observe that $\Psi(cx)=c^2\Psi(x)$ for $c\ge 0$. So $sx_0$ is a critical point of $\Psi$ for all $s\ge 0$. Let $\xi=[0,\infty)\cdot P^+ x_0\oplus E^0\oplus E^-$. 

We now claim that $\Psi(x)\le 0$ for all $x\in\xi$. To prove this, let $\xi_s=s P^+ x_0\oplus E^0\oplus E^-$. Observe that $\Psi|_{\xi_s}$ is a concave function. Since $sx_0$ is a critical point of $\Psi|_{\xi_s}$ it follows that $\max \Psi(\xi_s)=\Psi(sx_0)=s^2\Psi(x_0)=0$.

From \eqref{eq:ach}, \eqref{eq:f}, \eqref{eq:hr} and \eqref{eq:psi} we obtain
\[
\Ac_H(x)\le \Psi(x)+\alpha+\epsilon+(C-\alpha)\int_0^1 r(x(t))\,dt\le\alpha+\epsilon.
\]
Note that $\xi$ is $S^1$-invariant. Moreover it is proven in \cite{EH} that $h(\xi)\cap S^+\neq \emptyset$ for all $h\in\Gamma$. So $c_{H,1}\le\alpha+\epsilon$. Hence $c_1^{\op{EH}}(W)\le \alpha+\epsilon$ for all $\epsilon>0$. Therefore
\[
c_1^{\op{EH}}(W)\le \alpha.
\]

To prove the reverse inequality, recall from \cite[Prop.\ 2]{EH2} that $c_1^{\op{EH}}(W)$ is the symplectic action of some Reeb orbit on $\partial W$. Thus
\[
c_1^{\op{EH}}(W)\ge \alpha.
\]
\end{proof}

\end{document}

%% file: convextoric.pdf_tex
\begingroup%
  \makeatletter%
  \providecommand\color[2][]{%
    \errmessage{(Inkscape) Color is used for the text in Inkscape, but the package 'color.sty' is not loaded}%
    \renewcommand\color[2][]{}%
  }%
  \providecommand\transparent[1]{%
    \errmessage{(Inkscape) Transparency is used (non-zero) for the text in Inkscape, but the package 'transparent.sty' is not loaded}%
    \renewcommand\transparent[1]{}%
  }%
  \providecommand\rotatebox[2]{#2}%
  \newcommand*\fsize{\dimexpr\f@size pt\relax}%
  \newcommand*\lineheight[1]{\fontsize{\fsize}{#1\fsize}\selectfont}%
  \ifx\svgwidth\undefined%
    \setlength{\unitlength}{128.95602856bp}%
    \ifx\svgscale\undefined%
      \relax%
    \else%
      \setlength{\unitlength}{\unitlength * \real{\svgscale}}%
    \fi%
  \else%
    \setlength{\unitlength}{\svgwidth}%
  \fi%
  \global\let\svgwidth\undefined%
  \global\let\svgscale\undefined%
  \makeatother%
  \begin{picture}(1,1.00001776)%
    \lineheight{1}%
    \setlength\tabcolsep{0pt}%
    \put(0,0){\includegraphics[width=\unitlength,page=1]{convextoric.pdf}}%
    \put(0.24094589,0.24925444){\color[rgb]{0,0,0}\makebox(0,0)[lt]{\lineheight{1.25}\smash{\begin{tabular}[t]{l}$\Omega$\end{tabular}}}}%
  \end{picture}%
\endgroup%

%% file: concavetoric.pdf_tex
\begingroup%
  \makeatletter%
  \providecommand\color[2][]{%
    \errmessage{(Inkscape) Color is used for the text in Inkscape, but the package 'color.sty' is not loaded}%
    \renewcommand\color[2][]{}%
  }%
  \providecommand\transparent[1]{%
    \errmessage{(Inkscape) Transparency is used (non-zero) for the text in Inkscape, but the package 'transparent.sty' is not loaded}%
    \renewcommand\transparent[1]{}%
  }%
  \providecommand\rotatebox[2]{#2}%
  \newcommand*\fsize{\dimexpr\f@size pt\relax}%
  \newcommand*\lineheight[1]{\fontsize{\fsize}{#1\fsize}\selectfont}%
  \ifx\svgwidth\undefined%
    \setlength{\unitlength}{128.95602856bp}%
    \ifx\svgscale\undefined%
      \relax%
    \else%
      \setlength{\unitlength}{\unitlength * \real{\svgscale}}%
    \fi%
  \else%
    \setlength{\unitlength}{\svgwidth}%
  \fi%
  \global\let\svgwidth\undefined%
  \global\let\svgscale\undefined%
  \makeatother%
  \begin{picture}(1,1.00001776)%
    \lineheight{1}%
    \setlength\tabcolsep{0pt}%
    \put(0,0){\includegraphics[width=\unitlength,page=1]{concavetoric.pdf}}%
    \put(0.19109503,0.16201549){\color[rgb]{0,0,0}\makebox(0,0)[lt]{\lineheight{1.25}\smash{\begin{tabular}[t]{l}$\Omega$\end{tabular}}}}%
  \end{picture}%
\endgroup%

%% file: monotonetoric.pdf_tex
\begingroup%
  \makeatletter%
  \providecommand\color[2][]{%
    \errmessage{(Inkscape) Color is used for the text in Inkscape, but the package 'color.sty' is not loaded}%
    \renewcommand\color[2][]{}%
  }%
  \providecommand\transparent[1]{%
    \errmessage{(Inkscape) Transparency is used (non-zero) for the text in Inkscape, but the package 'transparent.sty' is not loaded}%
    \renewcommand\transparent[1]{}%
  }%
  \providecommand\rotatebox[2]{#2}%
  \newcommand*\fsize{\dimexpr\f@size pt\relax}%
  \newcommand*\lineheight[1]{\fontsize{\fsize}{#1\fsize}\selectfont}%
  \ifx\svgwidth\undefined%
    \setlength{\unitlength}{128.95602856bp}%
    \ifx\svgscale\undefined%
      \relax%
    \else%
      \setlength{\unitlength}{\unitlength * \real{\svgscale}}%
    \fi%
  \else%
    \setlength{\unitlength}{\svgwidth}%
  \fi%
  \global\let\svgwidth\undefined%
  \global\let\svgscale\undefined%
  \makeatother%
  \begin{picture}(1,1.00001776)%
    \lineheight{1}%
    \setlength\tabcolsep{0pt}%
    \put(0,0){\includegraphics[width=\unitlength,page=1]{monotonetoric.pdf}}%
    \put(0.34687902,0.22017482){\color[rgb]{0,0,0}\makebox(0,0)[lt]{\lineheight{1.25}\smash{\begin{tabular}[t]{l}$\Omega$\end{tabular}}}}%
  \end{picture}%
\endgroup%

%% file: wconvextoric.pdf_tex
\begingroup%
  \makeatletter%
  \providecommand\color[2][]{%
    \errmessage{(Inkscape) Color is used for the text in Inkscape, but the package 'color.sty' is not loaded}%
    \renewcommand\color[2][]{}%
  }%
  \providecommand\transparent[1]{%
    \errmessage{(Inkscape) Transparency is used (non-zero) for the text in Inkscape, but the package 'transparent.sty' is not loaded}%
    \renewcommand\transparent[1]{}%
  }%
  \providecommand\rotatebox[2]{#2}%
  \newcommand*\fsize{\dimexpr\f@size pt\relax}%
  \newcommand*\lineheight[1]{\fontsize{\fsize}{#1\fsize}\selectfont}%
  \ifx\svgwidth\undefined%
    \setlength{\unitlength}{128.95602856bp}%
    \ifx\svgscale\undefined%
      \relax%
    \else%
      \setlength{\unitlength}{\unitlength * \real{\svgscale}}%
    \fi%
  \else%
    \setlength{\unitlength}{\svgwidth}%
  \fi%
  \global\let\svgwidth\undefined%
  \global\let\svgscale\undefined%
  \makeatother%
  \begin{picture}(1,1.00001776)%
    \lineheight{1}%
    \setlength\tabcolsep{0pt}%
    \put(0,0){\includegraphics[width=\unitlength,page=1]{wconvextoric.pdf}}%
    \put(0.3572646,0.34480203){\color[rgb]{0,0,0}\makebox(0,0)[lt]{\lineheight{1.25}\smash{\begin{tabular}[t]{l}$\Omega$\end{tabular}}}}%
  \end{picture}%
\endgroup%

%% file: embeddings.pdf_tex
\begingroup%
  \makeatletter%
  \providecommand\color[2][]{%
    \errmessage{(Inkscape) Color is used for the text in Inkscape, but the package 'color.sty' is not loaded}%
    \renewcommand\color[2][]{}%
  }%
  \providecommand\transparent[1]{%
    \errmessage{(Inkscape) Transparency is used (non-zero) for the text in Inkscape, but the package 'transparent.sty' is not loaded}%
    \renewcommand\transparent[1]{}%
  }%
  \providecommand\rotatebox[2]{#2}%
  \newcommand*\fsize{\dimexpr\f@size pt\relax}%
  \newcommand*\lineheight[1]{\fontsize{\fsize}{#1\fsize}\selectfont}%
  \ifx\svgwidth\undefined%
    \setlength{\unitlength}{184.86328291bp}%
    \ifx\svgscale\undefined%
      \relax%
    \else%
      \setlength{\unitlength}{\unitlength * \real{\svgscale}}%
    \fi%
  \else%
    \setlength{\unitlength}{\svgwidth}%
  \fi%
  \global\let\svgwidth\undefined%
  \global\let\svgscale\undefined%
  \makeatother%
  \begin{picture}(1,0.72915768)%
    \lineheight{1}%
    \setlength\tabcolsep{0pt}%
    \put(0,0){\includegraphics[width=\unitlength,page=1]{embeddings.pdf}}%
    \put(0.25632024,0.01275753){\color[rgb]{0,0,0}\makebox(0,0)[lt]{\lineheight{1.25}\smash{\begin{tabular}[t]{l}$\mu_1$\end{tabular}}}}%
    \put(-0.00449023,0.35615801){\color[rgb]{0,0,0}\makebox(0,0)[lt]{\lineheight{1.25}\smash{\begin{tabular}[t]{l}$\mu_2$\end{tabular}}}}%
    \put(0.30848245,0.12432648){\color[rgb]{0,0,0}\makebox(0,0)[lt]{\lineheight{1.25}\smash{\begin{tabular}[t]{l}$\Delta^2(a)$\end{tabular}}}}%
    \put(0.51278397,0.23009966){\color[rgb]{0,0,0}\makebox(0,0)[lt]{\lineheight{1.25}\smash{\begin{tabular}[t]{l}$L(\mu_1,\mu_2)$\end{tabular}}}}%
    \put(0.39541917,0.24603788){\color[rgb]{0,0,0}\makebox(0,0)[lt]{\lineheight{1.25}\smash{\begin{tabular}[t]{l}$\Omega$\end{tabular}}}}%
  \end{picture}%
\endgroup%

%% file: weight.pdf_tex
\begingroup%
  \makeatletter%
  \providecommand\color[2][]{%
    \errmessage{(Inkscape) Color is used for the text in Inkscape, but the package 'color.sty' is not loaded}%
    \renewcommand\color[2][]{}%
  }%
  \providecommand\transparent[1]{%
    \errmessage{(Inkscape) Transparency is used (non-zero) for the text in Inkscape, but the package 'transparent.sty' is not loaded}%
    \renewcommand\transparent[1]{}%
  }%
  \providecommand\rotatebox[2]{#2}%
  \newcommand*\fsize{\dimexpr\f@size pt\relax}%
  \newcommand*\lineheight[1]{\fontsize{\fsize}{#1\fsize}\selectfont}%
  \ifx\svgwidth\undefined%
    \setlength{\unitlength}{118.28041711bp}%
    \ifx\svgscale\undefined%
      \relax%
    \else%
      \setlength{\unitlength}{\unitlength * \real{\svgscale}}%
    \fi%
  \else%
    \setlength{\unitlength}{\svgwidth}%
  \fi%
  \global\let\svgwidth\undefined%
  \global\let\svgscale\undefined%
  \makeatother%
  \begin{picture}(1,1.07849154)%
    \lineheight{1}%
    \setlength\tabcolsep{0pt}%
    \put(0,0){\includegraphics[width=\unitlength,page=1]{weight.pdf}}%
    \put(0.00485624,0.8442514){\color[rgb]{0,0,0}\makebox(0,0)[lt]{\lineheight{1.25}\smash{\begin{tabular}[t]{l}$b$\end{tabular}}}}%
    \put(0.83949943,0.02231384){\color[rgb]{0,0,0}\makebox(0,0)[lt]{\lineheight{1.25}\smash{\begin{tabular}[t]{l}$a$\end{tabular}}}}%
    \put(0,0){\includegraphics[width=\unitlength,page=2]{weight.pdf}}%
    \put(0.35360373,0.01993904){\color[rgb]{0,0,0}\makebox(0,0)[lt]{\lineheight{1.25}\smash{\begin{tabular}[t]{l}$\mu_1$\end{tabular}}}}%
    \put(-0.00701788,0.1784607){\color[rgb]{0,0,0}\makebox(0,0)[lt]{\lineheight{1.25}\smash{\begin{tabular}[t]{l}$\mu_2$\end{tabular}}}}%
    \put(0.00485624,0.46379959){\color[rgb]{0,0,0}\makebox(0,0)[lt]{\lineheight{1.25}\smash{\begin{tabular}[t]{l}$r$\end{tabular}}}}%
  \end{picture}%
\endgroup%

%% file: packing.pdf_tex
\begingroup%
  \makeatletter%
  \providecommand\color[2][]{%
    \errmessage{(Inkscape) Color is used for the text in Inkscape, but the package 'color.sty' is not loaded}%
    \renewcommand\color[2][]{}%
  }%
  \providecommand\transparent[1]{%
    \errmessage{(Inkscape) Transparency is used (non-zero) for the text in Inkscape, but the package 'transparent.sty' is not loaded}%
    \renewcommand\transparent[1]{}%
  }%
  \providecommand\rotatebox[2]{#2}%
  \newcommand*\fsize{\dimexpr\f@size pt\relax}%
  \newcommand*\lineheight[1]{\fontsize{\fsize}{#1\fsize}\selectfont}%
  \ifx\svgwidth\undefined%
    \setlength{\unitlength}{133.97340225bp}%
    \ifx\svgscale\undefined%
      \relax%
    \else%
      \setlength{\unitlength}{\unitlength * \real{\svgscale}}%
    \fi%
  \else%
    \setlength{\unitlength}{\svgwidth}%
  \fi%
  \global\let\svgwidth\undefined%
  \global\let\svgscale\undefined%
  \makeatother%
  \begin{picture}(1,0.9735198)%
    \lineheight{1}%
    \setlength\tabcolsep{0pt}%
    \put(0,0){\includegraphics[width=\unitlength,page=1]{packing.pdf}}%
    \put(-0.00619584,0.81338821){\color[rgb]{0,0,0}\makebox(0,0)[lt]{\lineheight{1.25}\smash{\begin{tabular}[t]{l}$a-r$\end{tabular}}}}%
    \put(0.11525507,0.72716283){\color[rgb]{0,0,0}\makebox(0,0)[lt]{\lineheight{1.25}\smash{\begin{tabular}[t]{l}$b$\end{tabular}}}}%
    \put(0.5132913,0.01097029){\color[rgb]{0,0,0}\makebox(0,0)[lt]{\lineheight{1.25}\smash{\begin{tabular}[t]{l}$r$\end{tabular}}}}%
  \end{picture}%
\endgroup%

%% file: example.pdf_tex
\begingroup%
  \makeatletter%
  \providecommand\color[2][]{%
    \errmessage{(Inkscape) Color is used for the text in Inkscape, but the package 'color.sty' is not loaded}%
    \renewcommand\color[2][]{}%
  }%
  \providecommand\transparent[1]{%
    \errmessage{(Inkscape) Transparency is used (non-zero) for the text in Inkscape, but the package 'transparent.sty' is not loaded}%
    \renewcommand\transparent[1]{}%
  }%
  \providecommand\rotatebox[2]{#2}%
  \newcommand*\fsize{\dimexpr\f@size pt\relax}%
  \newcommand*\lineheight[1]{\fontsize{\fsize}{#1\fsize}\selectfont}%
  \ifx\svgwidth\undefined%
    \setlength{\unitlength}{128.95602856bp}%
    \ifx\svgscale\undefined%
      \relax%
    \else%
      \setlength{\unitlength}{\unitlength * \real{\svgscale}}%
    \fi%
  \else%
    \setlength{\unitlength}{\svgwidth}%
  \fi%
  \global\let\svgwidth\undefined%
  \global\let\svgscale\undefined%
  \makeatother%
  \begin{picture}(1,1.09034539)%
    \lineheight{1}%
    \setlength\tabcolsep{0pt}%
    \put(0,0){\includegraphics[width=\unitlength,page=1]{example.pdf}}%
    \put(0.23679166,0.34373631){\color[rgb]{0,0,0}\makebox(0,0)[lt]{\lineheight{1.25}\smash{\begin{tabular}[t]{l}$\Omega_a$\end{tabular}}}}%
    \put(0,0){\includegraphics[width=\unitlength,page=2]{example.pdf}}%
    \put(0.81423104,0.03216832){\color[rgb]{0,0,0}\makebox(0,0)[lt]{\lineheight{1.25}\smash{\begin{tabular}[t]{l}$1$\end{tabular}}}}%
    \put(0.4237325,0.01543726){\color[rgb]{0,0,0}\makebox(0,0)[lt]{\lineheight{1.25}\smash{\begin{tabular}[t]{l}$1-2a$\end{tabular}}}}%
    \put(0,0){\includegraphics[width=\unitlength,page=3]{example.pdf}}%
  \end{picture}%
\endgroup%

%% file: czwt.pdf_tex
\begingroup%
  \makeatletter%
  \providecommand\color[2][]{%
    \errmessage{(Inkscape) Color is used for the text in Inkscape, but the package 'color.sty' is not loaded}%
    \renewcommand\color[2][]{}%
  }%
  \providecommand\transparent[1]{%
    \errmessage{(Inkscape) Transparency is used (non-zero) for the text in Inkscape, but the package 'transparent.sty' is not loaded}%
    \renewcommand\transparent[1]{}%
  }%
  \providecommand\rotatebox[2]{#2}%
  \newcommand*\fsize{\dimexpr\f@size pt\relax}%
  \newcommand*\lineheight[1]{\fontsize{\fsize}{#1\fsize}\selectfont}%
  \ifx\svgwidth\undefined%
    \setlength{\unitlength}{142.13589272bp}%
    \ifx\svgscale\undefined%
      \relax%
    \else%
      \setlength{\unitlength}{\unitlength * \real{\svgscale}}%
    \fi%
  \else%
    \setlength{\unitlength}{\svgwidth}%
  \fi%
  \global\let\svgwidth\undefined%
  \global\let\svgscale\undefined%
  \makeatother%
  \begin{picture}(1,0.97664783)%
    \lineheight{1}%
    \setlength\tabcolsep{0pt}%
    \put(0,0){\includegraphics[width=\unitlength,page=1]{czwt.pdf}}%
    \put(0.39933054,0.23331176){\color[rgb]{0,0,0}\makebox(0,0)[lt]{\lineheight{1.25}\smash{\begin{tabular}[t]{l}$\Omega$\end{tabular}}}}%
    \put(0,0){\includegraphics[width=\unitlength,page=2]{czwt.pdf}}%
    \put(0.8350526,0.7685138){\color[rgb]{0,0,0}\makebox(0,0)[lt]{\lineheight{1.25}\smash{\begin{tabular}[t]{l}$\mu_1=\mu_2$\end{tabular}}}}%
    \put(0,0){\includegraphics[width=\unitlength,page=3]{czwt.pdf}}%
    \put(0.73288964,0.01659256){\color[rgb]{0,0,0}\makebox(0,0)[lt]{\lineheight{1.25}\smash{\begin{tabular}[t]{l}$M_1$\end{tabular}}}}%
    \put(-0.00584003,0.70255575){\color[rgb]{0,0,0}\makebox(0,0)[lt]{\lineheight{1.25}\smash{\begin{tabular}[t]{l}$M_2$\end{tabular}}}}%
  \end{picture}%
\endgroup%

%% file: balls1.pdf_tex
\begingroup%
  \makeatletter%
  \providecommand\color[2][]{%
    \errmessage{(Inkscape) Color is used for the text in Inkscape, but the package 'color.sty' is not loaded}%
    \renewcommand\color[2][]{}%
  }%
  \providecommand\transparent[1]{%
    \errmessage{(Inkscape) Transparency is used (non-zero) for the text in Inkscape, but the package 'transparent.sty' is not loaded}%
    \renewcommand\transparent[1]{}%
  }%
  \providecommand\rotatebox[2]{#2}%
  \newcommand*\fsize{\dimexpr\f@size pt\relax}%
  \newcommand*\lineheight[1]{\fontsize{\fsize}{#1\fsize}\selectfont}%
  \ifx\svgwidth\undefined%
    \setlength{\unitlength}{146.68459832bp}%
    \ifx\svgscale\undefined%
      \relax%
    \else%
      \setlength{\unitlength}{\unitlength * \real{\svgscale}}%
    \fi%
  \else%
    \setlength{\unitlength}{\svgwidth}%
  \fi%
  \global\let\svgwidth\undefined%
  \global\let\svgscale\undefined%
  \makeatother%
  \begin{picture}(1,0.94479555)%
    \lineheight{1}%
    \setlength\tabcolsep{0pt}%
    \put(0,0){\includegraphics[width=\unitlength,page=1]{balls1.pdf}}%
    \put(-0.00418269,0.21055449){\color[rgb]{0,0,0}\makebox(0,0)[lt]{\lineheight{1.25}\smash{\begin{tabular}[t]{l}$a$\end{tabular}}}}%
    \put(0.05520446,0.79362827){\color[rgb]{0,0,0}\makebox(0,0)[lt]{\lineheight{1.25}\smash{\begin{tabular}[t]{l}1\end{tabular}}}}%
  \end{picture}%
\endgroup%

%% file: balls2.pdf_tex
\begingroup%
  \makeatletter%
  \providecommand\color[2][]{%
    \errmessage{(Inkscape) Color is used for the text in Inkscape, but the package 'color.sty' is not loaded}%
    \renewcommand\color[2][]{}%
  }%
  \providecommand\transparent[1]{%
    \errmessage{(Inkscape) Transparency is used (non-zero) for the text in Inkscape, but the package 'transparent.sty' is not loaded}%
    \renewcommand\transparent[1]{}%
  }%
  \providecommand\rotatebox[2]{#2}%
  \newcommand*\fsize{\dimexpr\f@size pt\relax}%
  \newcommand*\lineheight[1]{\fontsize{\fsize}{#1\fsize}\selectfont}%
  \ifx\svgwidth\undefined%
    \setlength{\unitlength}{152.62403559bp}%
    \ifx\svgscale\undefined%
      \relax%
    \else%
      \setlength{\unitlength}{\unitlength * \real{\svgscale}}%
    \fi%
  \else%
    \setlength{\unitlength}{\svgwidth}%
  \fi%
  \global\let\svgwidth\undefined%
  \global\let\svgscale\undefined%
  \makeatother%
  \begin{picture}(1,0.90995726)%
    \lineheight{1}%
    \setlength\tabcolsep{0pt}%
    \put(0,0){\includegraphics[width=\unitlength,page=1]{balls2.pdf}}%
    \put(0.09197164,0.76467272){\color[rgb]{0,0,0}\makebox(0,0)[lt]{\lineheight{1.25}\smash{\begin{tabular}[t]{l}1\end{tabular}}}}%
    \put(0,0){\includegraphics[width=\unitlength,page=2]{balls2.pdf}}%
    \put(0.49669285,0.00711762){\color[rgb]{0,0,0}\makebox(0,0)[lt]{\lineheight{1.25}\smash{\begin{tabular}[t]{l}$1-a$\end{tabular}}}}%
    \put(-0.00401992,0.4533487){\color[rgb]{0,0,0}\makebox(0,0)[lt]{\lineheight{1.25}\smash{\begin{tabular}[t]{l}$1-a$\end{tabular}}}}%
    \put(0.84174358,0.00711762){\color[rgb]{0,0,0}\makebox(0,0)[lt]{\lineheight{1.25}\smash{\begin{tabular}[t]{l}1\end{tabular}}}}%
  \end{picture}%
\endgroup%